\documentclass[10pt, reqno]{amsart}
\usepackage{amsmath, amsthm, amscd, amsfonts, amssymb, graphicx, color}
\usepackage[bookmarksnumbered, colorlinks, plainpages]{hyperref}
\hypersetup{colorlinks=true,linkcolor=red, anchorcolor=green, citecolor=cyan, urlcolor=red, filecolor=magenta, pdftoolbar=true}
\usepackage{mathrsfs}

\newtheorem{theorem}{Theorem}[section]
\newtheorem{lemma}[theorem]{Lemma}

\newtheorem{corollary}[theorem]{Corollary}
\theoremstyle{definition}

\newtheorem{example}[theorem]{Example}

\theoremstyle{remark}
\newtheorem{remark}[theorem]{Remark}
\numberwithin{equation}{section}

\begin{document}
\setcounter{page}{1}

\title[Remark on arithmetic topology]{Remark on arithmetic topology}

\author[I.Nikolaev]
{Igor Nikolaev$^1$}

\address{$^{1}$ Department of Mathematics and Computer Science, St.~John's University, 8000 Utopia Parkway,  
New York,  NY 11439, United States.}
\email{\textcolor[rgb]{0.00,0.00,0.84}{igor.v.nikolaev@gmail.com}}





\begin{abstract}
We formalize  the arithmetic topology, i.e. a  relationship between knots and primes.
   Namely, using the notion of a cluster 
$C^*$-algebra we construct  a functor from the category of 
3-dimensional manifolds $\mathscr{M}$  to a category of algebraic 
number fields $K$,  such that the prime ideals (ideals, resp.) in the ring of 
integers of $K$ correspond to knots (links, resp.) in  $\mathscr{M}$. 
It is proved that  the functor realizes all axioms of the 
arithmetic topology conjectured   in the 1960's by Manin, Mazur  and Mumford.  
\end{abstract}

\maketitle

\section{Introduction}
The famous Weil's Conjectures expose an amazing  link  between 
topology and number theory   [Weil 1949]  \cite{Wei1}. 
Likewise, the arithmetic topology studies an analogy between knots
and primes;  we refer the reader to  the book [Morishita 2012] \cite{M}
for an excellent introduction. 
To give an idea of the analogy, we quote [Mazur 1964] \cite{Maz1}: 

\medskip
\begin{minipage}[5cm]{12cm}
\textit{
``Guided by the results of Artin and Tate applied to the calculation of the Grothendieck Cohomology
Groups of the schemes:
\begin{equation}
Spec~\left(\mathbf{Z}/p\mathbf{Z}\right)\subset Spec~\mathbf{Z}
\end{equation}
Mumford has suggested a most elegant model as a geometric interpretation of the above
situation:  $Spec ~\left(\mathbf{Z}/p\mathbf{Z}\right)$ is like a one-dimensional knot in 
$Spec~\mathbf{Z}$ which is like a simply connected three-manifold.''
}
\end{minipage}

\bigskip\noindent
Roughly speaking, the idea is this. The $3$-dimensional sphere $\mathscr{S}^3$ corresponds 
to the field of rational numbers $\mathbf{Q}$. 
The prime ideals $p\mathbf{Z}$ in the ring of  integers $\mathbf{Z}$ correspond to
the knots $\mathscr{K}\subset \mathscr{S}^3$ and the ideals in $\mathbf{Z}$ correspond 
to the links $\mathscr{Z}\subset\mathscr{S}^3$. 
In general, a $3$-dimensional manifold $\mathscr{M}$ corresponds to an algebraic  number field $K$.
The prime ideals in the ring of integers $O_K$ of the field $K$ correspond to the knots  
$\mathscr{K}\subset \mathscr{M}$ and the ideals in $O_K$ correspond to the links 
 $\mathscr{Z}\subset\mathscr{M}$.

 \bigskip
 The aim of our note is   a functor from the category of closed 3-dimensional manifolds 
 to a category of algebraic number fields  realizing  all axioms of the arithmetic topology.
 The construction of such a functor is based on a representation of the braid group 
  into  a  cluster  $C^*$-algebra   \cite{Nik1}.  

\bigskip
Namely,  denote by $S_{g,n}$ a Riemann surface of genus $g$ with $n$ cusps.
Recall that  a {\it cluster algebra} $\mathscr{A}(\mathbf{x}, S_{g,n})$ of  the $S_{g,n}$  is a subring 
of the ring of the Laurent polynomials   $\mathbf{Z}[\mathbf{x}^{\pm 1}]$  
with integer coefficients and  variables $\mathbf{x}:=(x_1,\dots,x_{6g-6+2n})$.
The algebra $\mathscr{A}(\mathbf{x}, S_{g,n})$  is  a coordinate ring 
of the Teichm\"uller space $T_{g,n}$ of  surface $S_{g,n}$  [Williams 2014, Section 3]  \cite{Wil1}.  
The $\mathscr{A}(\mathbf{x}, S_{g,n})$   is a commutative algebra with an additive abelian
semigroup consisting of the Laurent polynomials with positive coefficients.  In particular,  the  $\mathscr{A}(\mathbf{x}, S_{g,n})$
is an abelian group with order satisfying the Riesz interpolation property, i.e. a 
 dimension group  [Effros 1981, Theorem 3.1]  \cite{E}. 
By a {\it cluster $C^*$-algebra}  $\mathbb{A}(\mathbf{x}, S_{g,n})$  we understand  an Approximately Finite 
$C^*$-algebra (AF-algebra)  such that $K_0(\mathbb{A}(\mathbf{x}, S_{g,n}))\cong  \mathscr{A}(\mathbf{x}, S_{g,n})$,
where $\cong$ is an isomorphism of the dimension groups.    We refer the reader to \cite{Nik1}  for 
the details and examples.   The  algebra  $\mathbb{A}(\mathbf{x}, S_{g,n})$ is a non-commutative coordinate
ring of the Teichm\"uller space $T_{g,n}$.    This observation  and the Birman-Hilden theorem imply a representation 
\begin{equation}\label{eq1.2}
\rho: B_{2g+n} \to   \mathbb{A}(\mathbf{x}, S_{g,n}),  \quad n\in \{0; 1\},
\end{equation}
 where $B_{2g+n}:=\{\sigma_1,\dots,\sigma_{2g+n-1} ~|~ \sigma_i\sigma_{i+1}\sigma_i=
 \sigma_{i+1}\sigma_i\sigma_{i+1},  ~\sigma_i\sigma_j=\sigma_j\sigma_i
 ~\hbox{if}  ~|i-j|\ge 2\}$ is the braid group, 
 the map $\rho$ acts  by the formula $\sigma_i\mapsto e_i+1$
and  $e_i$ are projections of the algebra $\mathbb{A}(\mathbf{x}, S_{g,n})$
 \cite{Nik2}.

 Let $b\in B_{2g+n}$ be a braid.  Denote by $\mathscr{L}_b$  a link obtained by the closure of $b$
 and let  $\pi_1(\mathscr{L}_b)$ be the  fundamental group of $\mathscr{L}_b$.
Recall that  
\begin{equation}\label{eq1.3}
\pi_1(\mathscr{L}_b)\cong\langle x_1,\dots, x_{2g+n} ~|~ x_i=r(b)x_i, ~1\le i\le 2g+n\rangle,
\end{equation}
where $x_i$ are generators of the free group $\mathbf{F}^{2g+n}$ and  $r: B_{2g+n}\to Aut~(\mathbf{F}^{2g+n})$ is 
the Artin representation of $B_{2g+n}$  [Artin 1925] \cite[Theorem 6]{Art1}.
 Let  $\mathcal{I}_b$ be  a two-sided   ideal  in the algebra  $\mathbb{A}(\mathbf{x}, S_{g,n})$
 generated by  relations (\ref{eq1.3}).   
In particular, the ideal  $\mathcal{I}_b$ is self-adjoint and representation  (\ref{eq1.2})  induces  a  representation 
\begin{equation}\label{eq1.4}
R: ~\pi_1(\mathscr{L}_b)\to \mathbb{A}(\mathbf{x}, S_{g,n})~/~\mathcal{I}_b
\end{equation}
on the quotient $\mathbb{A}(\mathbf{x}, S_{g,n})~/~\mathcal{I}_b:=\mathbb{A}_b$,  see lemma \ref{lm3.1}.   
The $\mathbb{A}_b$ is a stationary  AF-algebra of rank $6g-6+2n$;  we refer the reader to [Effros 1981]  \cite[Chapter 5]{E}
or Section 2.2 for the definitions.
It is known that the group $K_0(\mathbb{A}_b)\cong O_K$,  where $K$ is a number field of 
degree $6g-6+2n$  over $\mathbf{Q}$ [Effros 1981]  \cite[Chapter 5]{E}.

Thus we obtain  a map
$F: \mathcal{L}\to \mathscr{O}$,
where $\mathcal{L}$ is a category of all  links $\mathscr{L}$ modulo a homotopy
equivalence and  $\mathscr{O}$ is a category of  rings of the algebraic integers
$O_K$ modulo an isomorphism. The map $F$  acts by the formula:
\begin{equation}\label{eq1.6}
\mathscr{L}\buildrel\rm\pi_1\over\longmapsto \pi_1(\mathscr{L}_b)
\buildrel\rm R\over\longmapsto \mathbb{A}_b
\buildrel\rm K_0\over\longmapsto O_K,
\end{equation}
where  $R$ is given by  (\ref{eq1.4}).  
\begin{remark}
Using the Lickorish-Wallace Theorem [Lickorish 1962]  \cite{Lic1}, one can extend
the map $F$  to a category of 3-dimensional manifolds. Indeed,  recall that 
if $\mathscr{M}$ is a closed, orientable, connected 3-dimensional 
manifold, then there exists a link $\mathscr{L}\subset \mathscr{S}^3$ such that the Dehn 
surgery of $\mathscr{L}$ with the $\pm 1$ coefficients is homeomorphic to $\mathscr{M}$. 
(Notice that such a link  is not unique, but using the Kirby calculus one can define a 
canonical link $\mathscr{L}$ attached to $\mathscr{M}$.)  Thus  we get a 
map $\mathscr{M}\mapsto\mathscr{L}$. 
\end{remark}
\begin{theorem}\label{thm1.1}
The map $F$ is a  functor, such that:

\medskip
(i) $F(\mathscr{S}^3)=\mathbf{Z}$; 

\smallskip
(ii)  each  ideal $I\subseteq O_K=F(\mathscr{M})$ corresponds to 
a link $\mathscr{Z}\subset\mathscr{M}$; 

\smallskip
(iii)  each  prime ideal $I\subseteq O_K=F(\mathscr{M})$ corresponds to
a knot  $\mathscr{K}\subset\mathscr{M}$. 
\end{theorem}
The article is organized as follows. Section 2 contains a brief review
of braids, links and cluster $C^*$-algebras.  Theorem \ref{thm1.1} is proved 
in Section 3.  An illustration of theorem \ref{thm1.1} can be found in Section 4.

\section{Preliminaries}
A brief  review of  braids, links, AF-algebras and cluster $C^*$-algebras is given below. 
We refer the  reader to [Artin 1925] \cite{Art1},   [Effros 1981]  \cite{E}, 
[Morishita 2012] \cite{M},  [Williams 2014]  \cite{Wil1} and \cite{Nik1} for a detailed account.

\subsection{Braids, links and Galois covering}
By an $n$-string braid $b_n$ one understands two parallel copies of the plane $\mathbf{R}^2$ in 
$\mathbf{R}^3$  with $n$ distinguished points taken together with  $n$ disjoint smooth paths (``strings'')
joining pairwise the distinguished points of the planes; the tangent vector to each string
is never parallel to the planes. 
The braids $b$ are endowed with a natural equivalence relation: 
two braids $b$ and $b'$ are equivalent if $b$ can be deformed 
into $b'$ without intersection of the strings and so that at each moment of the
deformation $b$ remains  a braid.   
By an $n$-string braid group $B_n$ one understands the set of all $n$-string braids $b$
endowed with a multiplication operation of the concatenation of  $b\in B_n$ and $b'\in B_n$, 
i.e  the identification of the bottom of $b$ with the top of $b'$.  The group is non-commutative and the 
identity is given by the trivial braid. 
The  $B_n$ is isomorphic to a group on  generators 
$\sigma_1, \sigma_2,\dots,\sigma_{n-1}$  satisfying
the relations $\sigma_i\sigma_{i+1}\sigma_i  =\sigma_{i+1}\sigma_i\sigma_{i+1}$  and
$\sigma_i\sigma_j = \sigma_j\sigma_i$ if $|i-j|\ge 2$. 
By the {\it Artin representation}  we understand an injective homomorphism  $r: B_n\to Aut~(\mathbf{F}^n)$
into the group of automorphisms of the free group on generators $x_1,\dots, x_n$ given by the 
formula $\sigma_i: x_i\mapsto x_ix_{i+1} x_i^{-1}, ~\sigma_{i+1}: x_{i+1}\mapsto x_i$ and $\sigma_k=Id$ if
$k\ne i$ or $k\ne i+1$.

A closure of the braid $b$ is a link $\mathscr{L}_b\subset\mathbf{R}^3$
obtained by gluing the endpoints of strings at the top of the braid with such at the bottom
of the braid.  The closure of  two braids $b\in B_n$ and $b'\in B_m$ give the same link  $\mathscr{L}_b\subset\mathbf{R}^3$
 if and only if  $b$ and $b'$  can be connected by a 
sequence of the Markov moves of type I:  $b\mapsto aba^{-1}$ for a braid $a\in B_n$ and type  II:
 $b\mapsto b\sigma^{\pm 1}\in B_{n+1}$, 
where $\sigma\in B_{n+1}$.  
\begin{theorem}\label{thm2.1}
{\bf (\cite[Theorem 6]{Art1})}
$\pi_1(\mathscr{L}_b)\cong\langle x_1,\dots, x_n ~|~ x_1=r(b)x_1,\dots, x_n=r(b)x_n\rangle$,
where $x_i$ are generators of the free group $\mathbf{F}^n$ and  $r: B_n\to Aut~(\mathbf{F}^n)$ is 
the Artin representation of group $B_n$. 
\end{theorem}
Let $X$ be a topological space. A covering space of $X$ is a topological space 
 $X'$ and a continuous surjective map $p: X'\to X$ such that for an open neighborhood 
 $U$ of every point $x\in X$ the set $p^{-1}(U)$ is a union of disjoint open sets in $X$.   
A deck transformation of the covering space $X'$ is a homeomorphism $f:X'\to X'$ 
such that $p\circ f=p$. The set of all deck transformations is a group under composition 
denoted by $Aut~(X')$. The covering $p:X'\to X$ is called {\it Galois} (or regular) 
if the group   $Aut~(X')$ acts transitively on each fiber $p^{-1}(x)$, i.e. 
for any points $y_1,y_2\in p^{-1}(x)$ there exists $g\in Aut~(X')$ such that 
$y_2=g(y_1)$. The covering $p: X'\to X$ is Galois if and only if the group
$G:= p_{\ast}(\pi_1(X'))$ is a normal subgroup  of the fundamental group $\pi_1(X)$.  
In what follows we consider the Galois coverings of the space $X$ such that 
$|\pi_1(X)/G|<\infty$, i.e. the quotient $\pi_1(X)/G$ is a finite group.

\subsection{AF-algebras}
 A {\it $C^*$-algebra} is an algebra $A$ over $\mathbf{C}$ with a norm
$a\mapsto ||a||$ and an involution $a\mapsto a^*$ such that
it is complete with respect to the norm and $||ab||\le ||a||~ ||b||$
and $||a^*a||=||a^2||$ for all $a,b\in A$.
Any commutative $C^*$-algebra is  isomorphic
to the algebra $C_0(X)$ of continuous complex-valued
functions on some locally compact Hausdorff space $X$; 
otherwise, $A$ represents a noncommutative  topological
space.   For a unital $C^*$-algebra $A$, let $V(A)$
be the union (over $n$) of projections in the $n\times n$
matrix $C^*$-algebra with entries in $A$;
projections $p,q\in V(A)$ are  Murray - von Neumann  equivalent  if there exists a partial
isometry $u$ such that $p=u^*u$ and $q=uu^*$. The equivalence
class of projection $p$ is denoted by $[p]$;
the equivalence classes of orthogonal projections can be made to
a semigroup by putting $[p]+[q]=[p+q]$. The Grothendieck
completion of this semigroup to an abelian group is called
the  $K_0$-group of the algebra $A$.
The functor $A\to K_0(A)$ maps the category of unital
$C^*$-algebras into the category of abelian groups, so that
projections in the algebra $A$ correspond to a positive
cone  $K_0^+\subset K_0(A)$ and the unit element $1\in A$
corresponds to an order unit $u\in K_0(A)$.
The ordered abelian group $(K_0,K_0^+,u)$ with an order
unit  is called a {\it dimension group};  an order-isomorphism
class of the latter we denote by $(G,G^+)$.

An {\it $AF$-algebra}  $\mathbb{A}$  (Approximately Finite $C^*$-algebra) is defined to
be the  norm closure of an ascending sequence of   finite dimensional
$C^*$-algebras $M_n$,  where  $M_n$ is the $C^*$-algebra of the $n\times n$ matrices
with entries in $\mathbf{C}$.   Each embedding $M_n\to M_{n+1}$ is given by 
an integer non-negative matrix $A_n$.  An infinite graph given by the incidence matrices
$A_n$ is called a {\it Bratteli diagram} of the AF-algebra. The AF-algebra is defined by the 
Bratteli diagram.  If $A_n=Const$ for all $n$ the corresponding
AF-algebra is called {\it stationary}.  The rank of a stationary AF-algebra is defined 
as the rank of matrix $A_n$.  The dimension group
$(K_0(\mathbb{A}), K_0^+(\mathbb{A}), u)$ is a complete isomorphism
invariant of the algebra $\mathbb{A}$.  
The order-isomorphism  class $(K_0(\mathbb{A}), K_0^+(\mathbb{A}))$
 is an invariant of the Morita equivalence of algebra 
$\mathbb{A}$,  i.e.  an isomorphism class in the category of 
finitely generated projective modules over $\mathbb{A}$.    
The dimension group of any stationary AF-algebra has the 
form $(O_K, O_K^+, 1)$, where $O_K$ is the ring of integers 
of the number field $K$ generated by the Perron-Frobenius eigenvalue 
of matrix $A_n$,  $O_K^+$ consists of the positive elements of $O_K$ 
and $1$ is the rational unit.  The degree of $O_K$ over $\mathbf{Q}$ is equal to the 
rank of the corresponding stationary AF-algebra.

\subsection{Cluster $C^*$-algebras}
{\it Cluster algebra}  $\mathscr{A}(\mathbf{x}, B)$ of rank $n$ 
is a subring of the field  of  rational functions in $n$ variables
depending  on a  cluster  of variables  $\mathbf{x}=(x_1,\dots, x_n)$
and a skew-symmetric matrix  $B=(b_{ij})\in M_n(\mathbf{Z})$; 
the pair  $(\mathbf{x}, B)$ is called a  seed.
A new cluster $\mathbf{x}'=(x_1,\dots,x_k',\dots,  x_n)$ and a new
skew-symmetric matrix $B'=(b_{ij}')$ is obtained from 
$(\mathbf{x}, B)$ by the   exchange relations:
\begin{eqnarray}\label{eq2.1}
x_kx_k'  &=& \prod_{i=1}^n  x_i^{\max(b_{ik}, 0)} + \prod_{i=1}^n  x_i^{\max(-b_{ik}, 0)},\cr 
b_{ij}' &=& 
\begin{cases}
-b_{ij}  & \mbox{if}   ~i=k  ~\mbox{or}  ~j=k\cr
b_{ij}+{|b_{ik}|b_{kj}+b_{ik}|b_{kj}|\over 2}  & \mbox{otherwise.}
\end{cases}
\end{eqnarray}
The seed $(\mathbf{x}', B')$ is said to be a  mutation of $(\mathbf{x}, B)$ in direction $k$,
where $1\le k\le n$;   the algebra  $\mathscr{A}(\mathbf{x}, B)$ is  generated by cluster  variables $\{x_i\}_{i=1}^{\infty}$
obtained from the initial seed $(\mathbf{x}, B)$ by the iteration of mutations  in all possible
directions $k$.   The Laurent phenomenon  says  that  $\mathscr{A}(\mathbf{x}, B)\subset \mathbf{Z}[\mathbf{x}^{\pm 1}]$,
where  $\mathbf{Z}[\mathbf{x}^{\pm 1}]$ is the ring of  the Laurent polynomials in  variables $\mathbf{x}=(x_1,\dots,x_n)$
depending on an initial seed $(\mathbf{x}, B)$;  
in other words, each  generator $x_i$  of  algebra $\mathscr{A}(\mathbf{x}, B)$  can be 
written as a  Laurent polynomial in $n$ variables with the   integer coefficients.

Let $S_{g,n}$ be a Riemann surface of genus $g$ with $n$ cusps, such that $2g-2+n>0$.
Denote by   $T_{g,n}\cong \mathbf{R}^{6g-6+2n}$   the (decorated)  Teichm\"uller space of $S_{g,n}$,
i.e.  a collection   of all Riemann surfaces of genus  $g$ with $n$ cusps endowed with the natural topology  
[Penner 1987]  \cite{Pen1}.  In what follows,  we   focus on  the cluster algebras  $\mathscr{A}(\mathbf{x}, B)$,  
where matrix $B$  comes   from  an ideal   triangulation of the surface $S_{g,n}$ 
[Fomin, Shapiro \& Thurston 2008]   \cite{FoShaThu1}.
We denote  by  $\mathscr{A}(\mathbf{x}, S_{g,n})$ the corresponding cluster algebra of rank $6g-6+3n$. 
The $\mathscr{A}(\mathbf{x}, S_{g,n})$ is a coordinate ring of the space $T_{g,n}$ [Williams 2014]  \cite{Wil1}.

The  $\mathscr{A}(\mathbf{x}, S_{g,n})$  is a commutative algebra with an additive abelian
semigroup consisting of the Laurent polynomials with positive coefficients. 
Thus the algebra  $\mathscr{A}(\mathbf{x}, S_{g,n})$
is a countable abelian group with an order satisfying the Riesz interpolation property, i.e. a 
 dimension group  [Effros 1981, Theorem 3.1]  \cite{E}. 
A  {\it cluster $C^*$-algebra}  $\mathbb{A}(\mathbf{x}, S_{g,n})$  is   an  AF-algebra,  
such that $K_0(\mathbb{A}(\mathbf{x}, S_{g,n}))\cong  \mathscr{A}(\mathbf{x}, S_{g,n})$,
where $\cong$ is an isomorphism of the dimension groups \cite{Nik1}.
An element $e$ in a $C^*$-algebra is called a projection if $e^*=e=e^2$. 
\begin{theorem}\label{thm2.2}
{\bf (\cite{Nik2})}
The formula $\sigma_i\mapsto e_i+1$ defines a representation 
$\rho: B_{2g+n} \to   \mathbb{A}(\mathbf{x}, S_{g,n})$,  
where  $e_i$ are projections in the algebra $\mathbb{A}(\mathbf{x}, S_{g,n})$ and   $n\in \{0; 1\}$. 
\end{theorem}

\section{Proof of theorem \ref{thm1.1}}
We  shall split the proof in a series of lemmas. 
\begin{lemma}\label{lm3.1}
There exists a faithful representation $R: \pi_1(\mathscr{L}_b)\to \mathbb{A}_b$,
where  $\mathbb{A}_b$ is a stationary AF-algebra of rank $6g-6+2n$.  
\end{lemma}
\begin{proof}
(i) Let us construct a representation
\begin{equation}
R: ~\pi_1(\mathscr{L}_b)\to \mathbb{A}(\mathbf{x}, S_{g,n}) / \mathcal{I}_b:= \mathbb{A}_b.
\end{equation}
Suppose that $r: B_{2g+n}\to Aut ~(\mathbf{F}^{2g+n})$ is  the Artin representation of the braid group $B_{2g+n}$,
see Section 2.1. 
If $x_i$ is a generator of the free group $\mathbf{F}^{2g+n}$, one can think of $x_i$ as an element 
of the group  $Aut ~(\mathbf{F}^{2g+n})$ representing an automorphism of the left multiplication 
$\mathbf{F}^{2g+n}\to x_i \mathbf{F}^{2g+n}$. 
In view of Theorem \ref{thm2.1},  we get an embedding $\pi_1(\mathscr{L}_b)\hookrightarrow Aut ~(\mathbf{F}^{2g+n})$,
where $r(b)=Id$ is a trivial automorphism. 

Recall that the braid relations $\Sigma_i=\{x_i x_{i+1} x_i x_{i+1}^{-1} x_i^{-1} x_{i+1}^{-1}=Id,  
~x_i x_j\sigma_i^{-1} x_j^{-1}=Id~\hbox{if}  ~|i-j|\ge 2\}$ correspond to the trivial automorphisms of the 
group $\mathbf{F}^{2g+n}$. Thus 
\begin{equation}\label{eq3.2}
\pi_1(\mathscr{L}_b)\cong
\langle x_1,\dots, x_{2g+n} ~|~ r(b)=\Sigma_i=Id, ~1\le i\le 2g+n-1\rangle.
\end{equation}

Let now $\rho: B_{2g+n} \to   \mathbb{A}(\mathbf{x}, S_{g,n})$ be the representation constructed in 
Theorem \ref{thm2.2}. Because $B_{2g+n}\cong \langle x_i ~|~ \Sigma_i=Id, ~~1\le i\le 2g+n-1\rangle$ so that 
the ideal $\mathcal{I}_b\subset \mathbb{A}(\mathbf{x}, S_{g,n})$ is generated by the relation $r(b)=Id$,  one 
gets from (\ref{eq3.2}) a faithful  representation
\begin{equation}
R: ~\pi_1(\mathscr{L}_b)\to \mathbb{A}(\mathbf{x}, S_{g,n}) / \mathcal{I}_b.
\end{equation}

\bigskip
(ii) Let us show that the quotient  $\mathbb{A}_b= \mathbb{A}(\mathbf{x}, S_{g,n}) / \mathcal{I}_b$ 
is a stationary AF-algebra of rank $6g-6+2n$.  

First, let us show that the ideal $\mathcal{I}_b\subset \mathbb{A}(\mathbf{x}, S_{g,n})$ is self-adjoint,
i.e. $\mathcal{I}_b^*\cong \mathcal{I}_b$.  Indeed, the generating relation $r(b)=Id$ for such an ideal
is invariant under $\ast$-involution.  To prove the claim, we follow the argument and notation of 
\cite[Remark 4]{Nik2}. 
Namely, from Theorem \ref{thm2.2} the relation $r(b)=Id$ has  the form $(e_1+1)^{k_1}\dots (e_{n-1}+1)^{k_{n-1}}=1$,
where $k_i\in\mathbf{Z}$ and $e_i$ are projections.  Using the braid relations, one can  write the product at the LHS 
in the form $\sum_{i=1}^{|\mathcal{E}|} a_i\varepsilon_i$, where $a_i\in\mathbf{Z}$ and $\varepsilon_i$ are elements 
of a finite multiplicatively closed set $\mathcal{E}$.  Moreover, each $\varepsilon_i$ is (the Murray-von Neumann equivalent to)
a projection. In other words, $\varepsilon_i^*=\varepsilon_i$ and  
$(\sum_{i=1}^{|\mathcal{E}|} a_i\varepsilon_i)^*=\sum_{i=1}^{|\mathcal{E}|} a_i\varepsilon_i$.  
We conclude that the relation $r(b)=Id$ is invariant under $\ast$-involution. Therefore, we have 
$\mathcal{I}_b^*\cong \mathcal{I}_b$.

Recall that the quotient of the cluster $C^*$-algebra $\mathbb{A}(\mathbf{x}, S_{g,n})$ by a self-adjoint (primitive) 
ideal is a simple AF-algebra of rank $6g-6+2n$ \cite[Theorem 2]{Nik1}.  Let us show that 
the quotient   $\mathbb{A}(\mathbf{x}, S_{g,n}) / \mathcal{I}_b$ is a stationary  AF-algebra.

Consider an inner automorphism $\varphi_b$ of the group $B_{2g+n}$ given by the formula $x\mapsto b^{-1}xb$. 
Using the representation $\rho$ of Theorem \ref{thm2.2},  one  can extend the $\varphi_b$ to an automorphism of 
the algebra $\mathbb{A}(\mathbf{x}, S_{g,n})$.  Since the braid $b$ is a fixed point of $\varphi_b$,  we conclude that
$\varphi_b$ induces a non-trivial automorphism of the AF-algebra   $\mathbb{A}_b:=\mathbb{A}(\mathbf{x}, S_{g,n}) / \mathcal{I}_b$.
But each simple AF-algebra with a non-trivial group of automorphisms 
must be a stationary AF-algebra [Effros 1981]  \cite[Chapter 5]{E}.  Lemma \ref{lm3.1} follows. 
\end{proof}

\begin{lemma}\label{lm3.2}
There is  a one-to-one  correspondence between
normal subgroups of the group $\pi_1(\mathscr{L}_b)$ and
AF-subalgebras of the algebra    $\mathbb{A}_b$.  
\end{lemma}
\begin{proof}
Consider a representation $R: \pi_1(\mathscr{L}_b)\to \mathbb{A}_b$
constructed in lemma \ref{lm3.1}. 
The algebra $\mathbb{A}_b$ is a closure in the norm topology of a self-adjoint
representation of the group ring $\mathbf{C}[\pi_1(\mathscr{L}_b)]$ by bounded
linear operators acting on a Hilbert space. Namely, such a representation is given
by the formula $x_i\mapsto e_i+1$, where $x_i$ is a generator of the group $\pi_1(\mathscr{L}_b)$
and $e_i$ is a projection in the algebra $\mathbb{A}_b$. 

Let $G$ be a subgroup of  $\pi_1(\mathscr{L}_b)$. The $\mathbf{C}[G]$ is a subring 
of the group ring $\mathbf{C}[\pi_1(\mathscr{L}_b)]$. Taking the closure of a self-adjoint 
representation of $\mathbf{C}[G]$, one gets a $C^*$-subalgebra $\mathbb{A}_G$ of the algebra  $\mathbb{A}_b$. 

Let us show that if $G$ is a normal subgroup, then the $\mathbb{A}_G$ is an AF-algebra. 
Indeed, if $G$ is a normal subgroup one gets an exact sequence:
\begin{equation}\label{eq3.4}
0\to \mathbb{A}_G\to \mathbb{A}_b\to \pi_1(\mathscr{L}_b)/G\to 0,
\end{equation}
where $\pi_1(\mathscr{L}_b)/G$ is a finite group.  Let $\{M_k(\mathbf{C})\}_{k=1}^{\infty}$ 
be an ascending sequence of the finite-dimensional $C^*$-algebras, such that 
$\mathbb{A}_b=\lim_{k\to\infty} M_k(\mathbf{C})$.  Let $M_k'(\mathbf{C})=M_k(\mathbf{C})\cap \mathbb{A}_G$.
From (\ref{eq3.4}) we obtain an exact sequence:
\begin{equation}\label{eq3.5}
0\to M_k'(\mathbf{C})\to M_k(\mathbf{C})\to \pi_1(\mathscr{L}_b)/G\to 0. 
\end{equation}
Since  $|\pi_1(\mathscr{L}_b)/G|<\infty$, the $M_k'(\mathbf{C})$ is a finite-dimensional 
$C^*$-algebra. Thus $\mathbb{A}_G=\lim_{k\to\infty} M_k'(\mathbf{C})$, i.e. 
the $\mathbb{A}_G$ is an AF-algebra. Lemma \ref{lm3.2} follows.
\end{proof}

\begin{remark}\label{rmk3.3}
The $\mathbb{A}_G$ is a stationary AF-algebra, since it is an AF-subalgebra of a stationary 
AF-algebra  [Effros 1981]  \cite[Chapter 5]{E}.   
\end{remark}
\begin{corollary}\label{cor3.3}
There is  a one-to-one  correspondence between
normal subgroups of the group $\pi_1(\mathscr{L}_b)$
and  ideals in  the ring of integers $O_K\cong K_0(\mathbb{A}_b)$ 
of a number field $K$, where $\deg~(K|\mathbf{Q})=6g-6+2n$. 
\end{corollary}
\begin{proof}
A one-to-one correspondence between stationary AF-algebras and the rings of integers
in  number fields  has been established by  [Handelman 1981]  \cite{Han1}. 
Namely,  the dimension groups of stationary AF-algebras are one-to-one with the triples
$(\Lambda, [I], i)$, where $\Lambda\subset O_K$ is an order (i.e. a ring with the unit) in the number
field $K$,  $[I]$ is the equivalence class of ideals corresponding to $\Lambda$ and $i$ is the embedding 
class of the field $K$.  The degree of $K$ over $\mathbf{Q}$ is equal to the rank of stationary 
AF-algebra, i.e.  $\deg~(K|\mathbf{Q})=6g-6+2n$ by lemma \ref{lm3.1}.  

Assume for simplicity  that $\Lambda\cong O_K$, i.e. that $\Lambda$ is the maximal order in the field $K$.  
From (\ref{eq3.4}) we get an inclusion $K_0(\mathbb{A}_G)\subset K_0(\mathbb{A}_b)\cong O_K$. 
By remark \ref{rmk3.3},   the $K_0(\mathbb{A}_G)$ is an order in $O_K$.  
Since by (\ref{eq3.4}) the   $K_0(\mathbb{A}_G)$ is  the kernel of a homomorphism,   we conclude
that it  is an ideal in $O_K$. 
The rest of the proof follows from lemma \ref{lm3.2}.
Corollary \ref{cor3.3} is proved.
\end{proof}

\begin{lemma}\label{lm3.4}
Let $\mathscr{M}$  be a 3-dimensional manifold, such that $\pi_1(\mathscr{M})\cong\pi_1(\mathscr{L}_b)$
and let $O_K\cong K_0(\mathbb{A}_b)$.  
There is a one-to-one  correspondence between
the Galois coverings of $\mathscr{M}$ ramified over
a link $\mathscr{Z}\subset\mathscr{M}$ (a knot $\mathscr{K}\subset\mathscr{M}$, resp.)
and the ideals (the prime ideals, resp.) of the ring $O_K$. 
In other words, each link $\mathscr{Z}\hookrightarrow\mathscr{M}$ 
(each  knot $\mathscr{K}\hookrightarrow\mathscr{M}$, resp.)
corresponds to an ideal (a prime ideal, resp.)  of the ring $O_K$. 
\end{lemma}
\begin{proof}
Let 
\begin{equation}
\mathscr{Z}\cong \underbrace{S^1\cup S^1\cup\dots\cup S^1}_k
\end{equation}
and let $\mathscr{Z}\hookrightarrow\mathscr{M}$  be an embedding of link $\mathscr{Z}$ into
a 3-dimensional manifold $\mathscr{M}$. 
Let $\mathscr{M}_1$ be the Galois covering of $\mathscr{M}$ ramified over the first 
component $S^1$ of the link $\mathscr{Z}$ and such that the deck transformations fix the remaining 
components of $\mathscr{Z}$.  Let 
\begin{equation}
\mathscr{Z}_1\cong \underbrace{S^1\cup S^1\cup\dots\cup S^1}_{k-1}
\end{equation}
and let $\mathscr{Z}_1\hookrightarrow\mathscr{M}_1$  be an embedding of link $\mathscr{Z}_1$ into
 $\mathscr{M}_1$. Denote by $G_1$ a normal subgroup of $\pi_1(\mathscr{M})$ corresponding to the
 Galois covering $\mathscr{M}_1$. 
 
 Let $\mathscr{M}_2$ be the Galois covering of  $\mathscr{M}_1$ ramified over the first component of the 
 link $\mathscr{Z}_1$ such that  the remaining components are fixed by the corresponding deck transformations. 
 We  denote  by $G_2 \unlhd G_1$ a normal subgroup of $G_1$ corresponding to the Galois covering $\mathscr{M}_2$.  

Proceeding  by the induction, one gets  the following  lattice of the normal subgroups:
\begin{equation}\label{eq3.8}
G_k \unlhd G_{k-1}  \unlhd \dots  \unlhd G_1 \unlhd \pi_1(\mathscr{M}). 
\end{equation}
 By  corollary \ref{cor3.3},  the normal subgroups (\ref{eq3.8}) correspond to a 
 chain  of  ideals of the ring $O_K$:
\begin{equation}\label{eq3.9}
I_k \subset I_{k-1}  \subset \dots  \subset I_1 \subset  O_K. 
\end{equation}

\medskip
(i)  Suppose  that $k\ge 2$.   In this case the link  $\mathscr{Z}$ has at least two components, 
 and  i.e.  $\mathscr{Z}$ is distinct from a knot.   In view of (\ref{eq3.9}), the ideal $I_k$ cannot be a maximal ideal of the ring $O_K$, 
since $I_k\subset I_{k-1}\subset O_K$.

\medskip
(ii) Suppose that $k=1$.  In this case $\mathscr{Z}\cong \mathscr{K}$ is a knot. From (\ref{eq3.9}) the ideal $I_k$ is the
maximal ideal of the ring $O_K$.  Since  $O_K$ is the Dedekind domain, we conclude that $I_k$ is a prime ideal. 
This argument finishes the proof of lemma \ref{lm3.4}.
\end{proof}

\bigskip
Items (ii) and (iii) of theorem \ref{thm1.1} follow from lemma \ref{lm3.4}.

\bigskip
To prove item (i) of theorem  \ref{thm1.1},  one needs  to show that  $\mathscr{M}\cong \mathscr{S}^3$
implies $O_K\cong\mathbf{Z}$.   Indeed, consider the Riemann surface $S_{0,1}$, i.e.  sphere with a cusp.
Since the fundamental group $\pi_1(S_{0,1})$ is trivial, the 3-dimensional sphere $\mathscr{S}^3$
is homeomorphic to the mapping torus $\mathscr{M}_{\phi}$ of  surface $S_{0,1}$
by an automorphism  $\phi: S_{0,1}\to S_{0,1}$,  i.e.  $\mathscr{S}^3\cong \mathscr{M}_{\phi}$.

On the other hand,  the $S_{0,1}$ is homeomorphic to the interior of a planar $d$-gon. Thus  the 
cluster algebra   $\mathscr{A}(\mathbf{x}, S_{0,1})$  is isomorphic to the algebra $\mathscr{A}_{d-3}$
of the triangulated $d$-gon for $d\ge 3$, see  [Williams 2014]  \cite[Example 2.2]{Wil1}. 
It is known that the $\mathscr{A}_{d-3}$ is a cluster algebra of finite type, i.e. has finitely many seeds,  {\it ibid.}
Therefore the corresponding cluster $C^*$-algebra must be finite-dimensional,  i.e.  
$\mathbb{A}(\mathbf{x}, S_{0,1})\cong M_n(\mathbf{C})$. 
But $K_0(M_n(\mathbf{C})\cong \mathbf{Z}$ [Effros 1981]  \cite{E}. 
We conclude therefore that the the homeomorphism $\mathscr{M}\cong \mathscr{S}^3$
implies an isomorphism $O_K\cong\mathbf{Z}$.  This argument finishes the proof of 
item (i) of theorem \ref{thm1.1}. 

\bigskip
Theorem \ref{thm1.1} follows.

\section{Example}
Let   $g=n=1$, i.e. surface $S_{g,n}$ is homeomorphic to   the torus with a cusp. 
 The  matrix $B$ associated  to an ideal triangulation of  surface $S_{1,1}$    has the form: 
 \begin{equation}\label{eq4.1}
 B=\left(
 \begin{matrix}
 0 & 2 & -2\cr
              -2 & 0 & 2\cr
               2 & -2 & 0
 \end{matrix}
              \right)
 \end{equation}
[Fomin,  Shapiro  \& Thurston  2008]  \cite[Example 4.6]{FoShaThu1}. 
The  cluster $C^*$-algebra $\mathbb{A}(\mathbf{x}, S_{1,1})$ is given by the 
Bratteli diagram shown in Figure 1. 
Theorem \ref{thm2.2}  says that there exists a faithful representation of the braid group $B_3$:   
\begin{equation}\label{eq32}
 \rho: B_3\to \mathbb{A}(\mathbf{x}, S_{1,1}).
 \end{equation}
For every $b\in B_3$ the inner automorphism $\varphi_b: ~x\mapsto b^{-1}xb$ of $B_3$ defines a unique 
automorphism of the algebra $\mathbb{A}(\mathbf{x}, S_{1,1})$.
Since  the algebra $\mathbb{A}(\mathbf{x}, S_{1,1})$ is a coordinate ring
of the Teichm\"uller space $T_{1,1}$ and $Aut~(T_{1,1})\cong SL_2(\mathbf{Z})$,
we conclude that $\varphi_b$  corresponds  to an element of the modular group $SL_2(\mathbf{Z})$.

An explicit formula for the correspondence $b\mapsto \varphi_b$ is well known.
Namely, if $\sigma_1$ and $\sigma_2$ are the standard generators of the braid group
$B_3\cong \langle \sigma_1, \sigma_2 ~|~ \sigma_1\sigma_2\sigma_1=\sigma_2\sigma_1\sigma_2\rangle$
then the formula
\begin{equation}
\sigma_1\mapsto\left(
\begin{matrix}
1 & 1\\
0 & 1
\end{matrix}
\right), 
\qquad 
\sigma_2\mapsto\left(
\begin{matrix}
1 & 0\\
-1 & 1
\end{matrix}
\right) 
\end{equation}
defines a surjective homomorphism $B_3\to SL_2(\mathbf{Z})$.

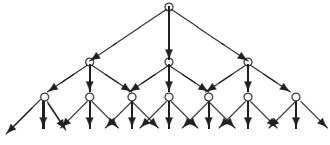
\begin{figure}
\begin{picture}(100,100)(0,130)

\put(50,200){\circle{3}}
\put(50,179){\circle{3}}
\put(20,179){\circle{3}}
\put(80,179){\circle{3}}


\put(3,166){\circle{3}}
\put(20,166){\circle{3}}
\put(36,166){\circle{3}}
\put(50,166){\circle{3}}
\put(65,166){\circle{3}}

\put(80,166){\circle{3}}
\put(98,166){\circle{3}}


\put(49,199){\vector(-3,-2){29}}
\put(51,199){\vector(3,-2){29}}
\put(50,200){\vector(0,-1){20}}


\put(19,178){\vector(-3,-2){15}}
\put(21,178){\vector(3,-2){15}}
\put(20,178){\vector(0,-1){10}}


\put(50,178){\vector(-3,-2){15}}
\put(50,178){\vector(3,-2){15}}
\put(50,178){\vector(0,-1){10}}


\put(79,178){\vector(-3,-2){15}}
\put(81,178){\vector(3,-2){15}}
\put(80,178){\vector(0,-1){10}}


\put(1.5,165){\vector(-1,-1){13}}
\put(4,164){\vector(2,-3){7}}
\put(2.5,164){\vector(0,-1){10}}

\put(19,164){\vector(-1,-1){10}}
\put(21,164){\vector(1,-1){10}}
\put(20,164){\vector(0,-1){10}}

\put(49,164){\vector(-1,-1){10}}
\put(51,164){\vector(1,-1){10}}
\put(50,164){\vector(0,-1){10}}

\put(79,164){\vector(-1,-1){10}}
\put(81,164){\vector(1,-1){10}}
\put(80,164){\vector(0,-1){10}}

\put(98,164){\vector(-1,-1){10}}
\put(100,164){\vector(1,-1){10}}
\put(98,164){\vector(0,-1){10}}


\put(65,164){\vector(-1,-1){10}}
\put(65,164){\vector(1,-1){10}}
\put(65,164){\vector(0,-1){10}}

\put(36,164){\vector(-1,-1){10}}
\put(36,164){\vector(1,-1){10}}
\put(36,164){\vector(0,-1){10}}


\end{picture}
\caption{Bratteli diagram of the algebra $\mathbb{A}(\mathbf{x}, S_{1,1})$.} 
\end{figure}

\begin{example}
Let $\mathscr{L}_b$ be a link given by the closure of a braid of the form $b=\sigma_1^p\sigma_2^{-q}$,
where $p\ge 1$ and $q\ge 1$.  In this case 
\begin{equation}\label{eq4.4}
\varphi_b=
\left(
\begin{matrix}
pq+1 & p\\
q & 1
\end{matrix}
\right).  
\end{equation}
The algebra $\mathbb{A}_b=\mathbb{A}(\mathbf{x}, S_{1,1})/\mathcal{I}_b$
is a stationary AF-algebra of rank $2$  given by the Bratteli diagram shown in Figure 2. 
(The numbers in Figure 2 show the multiplicity of the corresponding edge of the graph.) 
The Perron-Frobenius eigenvalue of the matrix  $\varphi_b$ is equal to
\begin{equation}
\lambda_{\varphi_b}={pq+2+\sqrt{pq(pq+4)}\over 2}.
\end{equation}
Therefore $K_0(\mathbb{A}_b)\cong \mathbf{Z}+\mathbf{Z}\sqrt{D}$, where $D=pq(pq+4)$. 
The number  field $K$ corresponding to the link $\mathscr{L}_b$ is a real quadratic field of the form:
\begin{equation}
K\cong \mathbf{Q}\left(\sqrt{pq(pq+4)}\right).
\end{equation}
Denote by $\mathscr{M}_{p,q}$  a 3-dimensional manifold, such that $\pi_1(\mathscr{M}_{p,q})
\cong\pi_1(\mathscr{L}_{\sigma_1^p\sigma_2^{-q}})$.      
The manifolds  $\mathscr{M}_{p,q}$ corresponding to the quadratic fields with a  small square-free discriminant $D$
are recorded below.
\end{example}
\begin{table}[h]
\begin{tabular}{c|c}
\hline
&\\
Manifold  $\mathscr{M}_{p,q}$ & Number field $K=F(\mathscr{M}_{p,q})$\\
&\\
\hline
 $\mathscr{M}_{1,1}$ & $\mathbf{Q}(\sqrt{5})$ \\
\hline
 $\mathscr{M}_{1,3}$ & $\mathbf{Q}(\sqrt{21})$ \\
\hline
 $\mathscr{M}_{1,7}$ & $\mathbf{Q}(\sqrt{77})$ \\
\hline
 $\mathscr{M}_{1,11}$ & $\mathbf{Q}(\sqrt{165})$ \\
\hline
 $\mathscr{M}_{1,13}$ & $\mathbf{Q}(\sqrt{221})$ \\
\hline
 $\mathscr{M}_{3,5}$ & $\mathbf{Q}(\sqrt{285})$ \\
\hline
 $\mathscr{M}_{3,7}$ & $\mathbf{Q}(\sqrt{525})$ \\
\hline
 $\mathscr{M}_{3,11}$ & $\mathbf{Q}(\sqrt{1221})$ \\
\hline
\end{tabular}
\end{table}
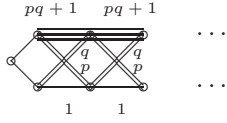
\begin{figure}
\begin{picture}(300,60)(0,0)
\put(110,30){\circle{3}}
\put(120,20){\circle{3}}
\put(140,20){\circle{3}}
\put(160,20){\circle{3}}
\put(120,40){\circle{3}}
\put(140,40){\circle{3}}
\put(160,40){\circle{3}}

\put(110,30){\line(1,1){10}}
\put(110,30){\line(1,-1){10}}
\put(120,42){\line(1,0){20}}
\put(120,40){\line(1,0){20}}
\put(120,38){\line(1,0){20}}
\put(120,41){\line(1,-1){20}}
\put(120,39){\line(1,-1){20}}
\put(120,19){\line(1,1){20}}
\put(120,21){\line(1,1){20}}
\put(140,42){\line(1,0){20}}
\put(140,40){\line(1,0){20}}
\put(140,38){\line(1,0){20}}
\put(140,41){\line(1,-1){20}}
\put(140,39){\line(1,-1){20}}
\put(140,19){\line(1,1){20}}
\put(140,21){\line(1,1){20}}
\put(120,20){\line(1,0){20}}
\put(140,20){\line(1,0){20}}

\put(180,20){$\dots$}
\put(180,40){$\dots$}

\put(115,48){\tiny $pq+1$}
\put(145,48){\tiny $pq+1$}

\put(130,10){\tiny $1$}
\put(150,10){\tiny $1$}

\put(136,26){\tiny $p$}
\put(156,26){\tiny $p$}

\put(136,32){\tiny $q$}
\put(156,32){\tiny $q$}

\end{picture}

\caption{Bratteli diagram of the algebra   $\mathbb{A}(\mathbf{x}, S_{1,1})/\mathcal{I}_b$.}
\end{figure}
\begin{remark}
The manifold $\mathscr{M}_{p,q}$ can be realized as a torus bundle over the circle with the monodromy 
given by matrix (\ref{eq4.4}).  The arithmetic invariants of surface   bundles over the circle  were studied in \cite{Nik3}. 
\end{remark}

\bibliographystyle{amsplain}


\end{document}